\documentclass[a4paper]{article}

\usepackage{fullpage}
\usepackage{tikz}

\usepackage[pdftex,breaklinks,colorlinks,
linkcolor=blue,
citecolor=blue,
urlcolor=blue]{hyperref}

\usepackage{enumitem}

\usepackage{amsmath,amssymb,amsfonts,amsthm}

\newtheorem{theorem}{Theorem}
\newtheorem{conjecture}[theorem]{Conjecture}
\newtheorem{lemma}[theorem]{Lemma}
\newtheorem{proposition}[theorem]{Proposition}
\newtheorem{corollary}[theorem]{Corollary}

\theoremstyle{definition}
\newtheorem{definition}[theorem]{Definition}
\newtheorem{remark}[theorem]{Remark}
\newtheorem*{notation}{Notation}

\author{Pablo Romero\footnote{Facultad de Ingenier\'ia, Universidad de la Rep\'ublica, Montevideo, Uruguay. E-mail address: \texttt{promero@fing.edu.uy}}\qquad Mart\'\i n D.\ Safe\footnote{Departamento de Matem\'atica, Universidad Nacional del Sur (UNS), Bah\'ia Blanca, Argentina and INMABB, Universidad Nacional del Sur (UNS)-CONICET, Bah\'ia Blanca, Argentina. E-mail address: \texttt{msafe@uns.edu.ar}}}

\date{}

\makeatletter%
\begin{document}

\title{Nonexistence of uniformly most reliable graphs of least corank}

\maketitle

\begin{abstract}
If $G$ is a simple graph and $\rho\in[0,1]$, the reliability $R_G(\rho)$ is the probability of $G$ being connected after each of its edges is removed independently with probability $\rho$. A simple graph $G$ is a \emph{uniformly most reliable graph} (UMRG) if $R_G(\rho)\geq R_H(\rho)$ for every $\rho\in[0,1]$ and every simple graph $H$ on the same number of vertices and edges as $G$. Boesch [J.\ Graph Theory 10 (1986), 339--352] conjectured that, if $n$ and $m$ are such that there exists a connected simple graph on $n$ vertices and $m$ edges, then there also exists a UMRG on the same number of vertices and edges. Some counterexamples to Boesch's conjecture were given by Kelmans, Myrvold et al., and Brown and Cox. It is known that Boesch's conjecture holds whenever the corank, defined as $c=m-n+1$, is at most $4$ (and the corresponding UMRGs are fully characterized). Ath and Sobel conjectured that Boesch's conjecture holds whenever the corank $c$ is between $5$ and $8$, provided the number of vertices is at least $2c-2$. In this work, we give an infinite family of counterexamples to Boesch's conjecture of corank $5$. These are the first reported counterexamples that attain the minimum possible corank. As a byproduct, the conjecture by Ath and Sobel is disproved.
\end{abstract}

\renewcommand{\labelitemi}{--}

\section{Introduction}\label{section:motivation}

If $G$ is a simple graph, the \emph{reliability of $G$ with failure probability $\rho\in [0,1]$}, denoted $R_G(\rho)$, is the probability of $G$ being connected after each of its edges is removed independently with probability $\rho$. Given integers $n$ and $m$, the question arises as to whether there exists a simple graph $G$ on $n$ vertices and $m$ edges such that the reliability $R_G(\rho)$ is greater than or equal to the reliability $R_H(\rho)$ for every simple graph $H$ on $n$ vertices and $m$ edges and every $\rho\in[0,1]$. Such a simple graph $G$ is called a \emph{uniformly most reliable graph} (UMRG). This concept was introduced in 1986 by Boesch in his seminal article~\cite{1986-Boesch}. 

If all the simple graphs on $n$ vertices and $m$ edges are disconnected, then all of them are UMRGs (for all of them have reliability $0$, irrespective of $\rho$). Thus, we can restrict our attention to the case where the class $\mathcal C_{n,m}$ of connected simple graphs on $n$ vertices and $m$ edges is nonempty. If $\mathcal{C}_{n,m}$ is nonempty, we define the \emph{corank} of $\mathcal C_{n,m}$, and of each of the graphs in $\mathcal C_{n,m}$, as $m-n+1$.

The study of UMRGs with small corank was pioneered by Boesch et al.\ \cite{1991-Boesch}. They observed that all trees, all cycles, and all the so-called balanced $\theta$-graphs (see the definition in Section~\ref{section:basics}) are all the UMRGs having corank equal to $0$, $1$, and $2$, respectively. Moreover, the authors also proved that all UMRGs of corank $3$ are certain subdivisions of the complete graph $K_4$. In the same work, it was conjectured that all UMRGs of corank $4$ are the $4$-wheel and certain subdivisions of the complete bipartite graph $K_{3,3}$. This conjecture was proved by Wang~\cite{1994-Wang}. The existence of UMRGs in classes $\mathcal C_{n,m}$ of corank $5$ is only reported in the literature for $n \in \{5,6,7,8\}$. The corresponding UMRGs are the three graphs depicted in Figure~\ref{figure:primeros4} together with the Wagner graph $W$ depicted in Figure~\ref{figure:cubicosA} and were found by Myrvold~\cite{1990-Myrvold} by exhaustive computer search on all simple graphs on up to $8$ vertices. At the other end of the spectrum, regarding dense graphs, Kelmans~\cite{1981-Kelmans} proved the existence of UMRGs in all the classes $\mathcal C_{n,m}$ such that $n\geq 3$ and $\binom n2-\lfloor\frac n2\rfloor \leq m \leq \binom{n}{2}
$. The corresponding UMRGs arise by removing a matching from the complete graph $K_n$. Recently, Archer et al.~\cite{2019-Archer} established the existence of UMRGs also in the classes $\mathcal C_{n,m}$ for each odd $n\geq 5$ and $m$ either $\binom n2-\frac{n+1}2$ or $\binom n2-\frac{n+3}2$.

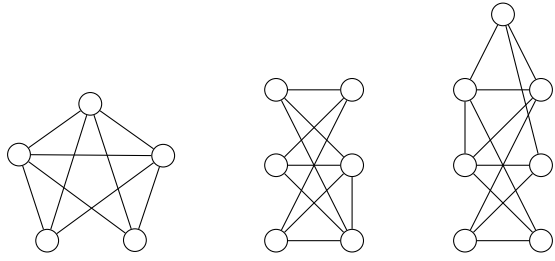
\begin{figure}
\begin{center}
\scalebox{1}{
\begin{tikzpicture}[ scale=1, nodo/.style={circle,draw=black!120,fill=white!120,inner sep=0pt,minimum size=3mm},rotate=18]
     
\node[nodo] (1) at (1,0) {};
\node[nodo] (2) at (0.3,0.95) {};
\node[nodo] (3) at (-0.8,0.6) {};
\node[nodo] (4) at (-0.8,-0.6) {};
\node[nodo] (5) at (0.3,-0.95) {};

\draw (1) to (2);
\draw (1) to (3);
\draw (1) to (4);
\draw (1) to (5);
\draw (2) to (3);
\draw (2) to (4);
\draw (2) to (5);
\draw (3) to (4);
\draw (3) to (5);

\end{tikzpicture} \qquad\quad
\begin{tikzpicture}[ scale=1, nodo/.style={circle,draw=black!120,fill=white!120,inner sep=0pt,minimum size=3mm},rotate=90]
     
\node[nodo] (1) at (-1,0) {};
\node[nodo] (2) at (0,0) {};
\node[nodo] (3) at (1,0) {};
\node[nodo] (4) at (-1,1) {};
\node[nodo] (5) at (0,1) {};
\node[nodo] (6) at (1,1) {};

\draw (1) to (4);
\draw (1) to (5);
\draw (1) to (6);
\draw (2) to (4);
\draw (2) to (5);
\draw (2) to (6);
\draw (3) to (4);
\draw (3) to (5);
\draw (3) to (6);
\draw (1) to (2);

\end{tikzpicture} \qquad\quad
\begin{tikzpicture}[ scale=1, nodo/.style={circle,draw=black!120,fill=white!120,inner sep=0pt,minimum size=3mm},rotate=-90]
     
\node[nodo] (1) at (-1,0) {};
\node[nodo] (2) at (0,0) {};
\node[nodo] (3) at (1,0) {};
\node[nodo] (4) at (-1,1) {};
\node[nodo] (5) at (0,1) {};
\node[nodo] (6) at (1,1) {};
\node[nodo] (7) at (-2,0.5) {};

\draw (1) to (4);
\draw (1) to (7);
\draw (7) to (5);
\draw (7) to (4);
\draw (1) to (6);
\draw (2) to (4);
\draw (2) to (5);
\draw (2) to (6);
\draw (3) to (4);
\draw (3) to (5);
\draw (3) to (6);
\draw (1) to (2);

\end{tikzpicture} }
\end{center}
\caption{Small UMRGs of corank $5$\label{figure:primeros4}} 
\end{figure}

In~\cite{1986-Boesch}, Boesch conjectured that the following is true.

\begin{conjecture}[Boesch \cite{1986-Boesch}]\label{conjecture:boesch}
If $\mathcal C_{n,m}$ is nonempty, then it contains at least one UMRG.
\end{conjecture}

The characterizations of all the UMRGs up to corank $4$ proved in \cite{1991-Boesch} and \cite{1994-Wang}, discussed above, imply that Boesch's conjecture holds for all the classes $\mathcal C_{n,m}$ having corank at most $4$. Thus, any counterexample to Conjecture~\ref{conjecture:boesch} must have corank at least $5$.

Infinitely many counterexamples to Conjecture~\ref{conjecture:boesch} are known~\cite{1981-Kelmans,1991-Myrvold,2014-Brown}. These counterexamples consist of: (i) the class $\mathcal C_{6,11}$, which has corank $6$; (ii) the class $\mathcal C_{7,15}$ having corank $9$; (iii) the classes $\mathcal C_{8,14}$, $\mathcal C_{8,17}$, $\mathcal C_{8,19}$, $\mathcal C_{8,22}$, and $\mathcal C_{8,23}$, having coranks $7$, $10$, $12$, $15$, and $16$, respectively; and (iv) infinitely many other counterexamples $\mathcal C_{n,m}$ with $n\geq 9$, whose coranks are at least $21$ (and grow asymptotically as $n^2/2$). As a result, all known counterexamples to Conjecture~\ref{conjecture:boesch} have corank at least $6$ and it is currently not known whether there are counterexamples of corank $5$. 

Ath and Sobel~\cite{2000-Ath} proposed the following weaker conjecture in 2000.

\begin{conjecture}[Ath-Sobel~\cite{2000-Ath}]\label{conjecture:ath-sobel}
If a nonempty class $\mathcal C_{n,m}$ has corank $c\in\{5,6,7,8\}$ and $n\geq 2c-2$, then $\mathcal C_{n,m}$ contains at least one UMRG.
\end{conjecture}

Furthermore, explicit candidates for such UMRGs were also proposed in~\cite{2000-Ath}. 

Our main result is the nonexistence of UMRGs in the classes $\mathcal C_{n,m}$ such that $n=12s+4$ and having corank $5$ (i.e., $m=12s+8$), for every positive integer $s$. Thereby, we provide an infinite family of counterexamples to Conjecture~\ref{conjecture:boesch} of corank $5$. These are the first known counterexamples which attain the minimum possible corank. As a byproduct, the weaker Conjecture~\ref{conjecture:ath-sobel} proposed by Ath and Sobel is disproved.

Our proof strategy consists of two steps. First, for each positive integer $s$, we determine a simple graph $G$ (depending on $s$) in the class $\mathcal C_{12s+4,12s+8}$ whose reliability is strictly greater than that of all other graphs in that same class whenever $\rho \in (0,\epsilon)$, for some $\epsilon>0$. Second, we find a simple graph in $\mathcal C_{12s+4,12s+8}$ whose reliability is strictly greater than that of $G$ whenever $\rho \in (1-\delta,1)$, for some $\delta>0$. Since a UMRG must attain the greatest possible reliability for each $\rho \in [0,1]$, the nonexistence of a UMRG in the class $\mathcal C_{12s+4,12s+8}$ follows.

This article is organized as follows. Section~\ref{section:basics} presents basic graph-theoretic terminology. Section~\ref{section:background} discusses in greater detail the concept of uniformly most reliable graphs and the previously known results. Some preliminary results are proved in Section~\ref{section:preliminary}. In Section~\ref{section:lmrg}, we identify the locally most reliable graph near $\rho=0$ in the class $\mathcal C_{12s+4,12s+8}$, for each positive integer $s$. The main result is given in Section~\ref{section:main}. 

\section{Basic definitions}\label{section:basics}

This section introduces the basic definitions. More specific definitions will be given throughout the article. If $c$ is a positive integer, $[c]$ denotes the set $\{1,2,\ldots,c\}$. We denote the set of positive integers by $\mathbb Z^+$. Let $S$ be a finite set. The cardinality of $S$ is denoted by $\vert S\vert$. We also refer to the cardinality of a set as its {size}. If $k$ is a nonnegative integer, then the family of all the subsets of $S$ with cardinality $k$ is denoted by $\binom{S}{k}$.

All the graphs in this work are finite and undirected. We denote the vertex set and the edge set of a graph $G$ by $V(G)$ and $E(G)$, respectively. A graph is \emph{simple} if it has no parallel edges nor loops. Let $G$ be a graph. If $S\subseteq V(G)$, we denote by $G-S$ the graph that arises from $G$ by removing all vertices in $S$ and by $G[S]$ the \emph{subgraph of $G$ induced by $S$} (i.e.\ the graph $G-(V(G)-S)$). If $F\subseteq E(G)$, $G-F$ denotes the graph with vertex set $V(G)$ and edge set $E(G)-F$. We denote the chordless path, the chordless cycle, and the complete graph on $n$ vertices by $P_n$, $C_n$, and $K_n$, respectively. The \emph{4-wheel} is the simple graph that arises from $C_4$ by adding one vertex adjacent to every other vertex. The \emph{Wagner graph} $W$ and the \emph{cube} $Q$ are depicted in Figure~\ref{figure:cubicosA}. A \emph{$\theta$-graph} is a simple graph consisting precisely of three paths having the same two endpoints, whose set of internal vertices are pairwise disjoint, and such that the subgraph induced by the vertices of each two of these paths induces a chordless cycle. If the lengths of each two of these three paths differ by at most $1$, the graph is called a \emph{balanced $\theta$-graph}.

\begin{figure}
\begin{center}
\scalebox{0.8}{
\begin{tabular}{c}
\begin{tikzpicture}[ scale=1, nodo/.style={circle,draw=black!120,fill=white!120,inner sep=0pt,minimum size=5mm}]

\node[nodo] (1) at (0,2) {$1$};
\node[nodo] (2) at (1.4142,1.4142) {$2$};
\node[nodo] (3) at (2,0) {$3$};
\node[nodo] (4) at (1.4142,-1.4142) {$4$};
\node[nodo] (5) at (0,-2) {$5$};
\node[nodo] (6) at (-1.4142,-1.4142) {$6$};
\node[nodo] (7) at (-2,0) {$7$};
\node[nodo] (8) at (-1.4142,1.4142) {$8$};

\draw (1) to (2);
\draw (2) to (3);
\draw (3) to (4);
\draw (4) to (5);
\draw (5) to (6);
\draw (6) to (7);
\draw (7) to (8);
\draw (8) to (1);
\draw (1) to (5);
\draw (2) to (6);
\draw (3) to (7);
\draw (4) to (8);

\end{tikzpicture} %
\\
$W$
\end{tabular}
\qquad\hfill\qquad
\begin{tabular}{c}
\begin{tikzpicture}[ scale=1, nodo/.style={circle,draw=black!120,fill=white!120,inner sep=0pt,minimum size=5mm}]

\node[nodo] (1) at (0,2) {$1$};
\node[nodo] (2) at (1.4142,1.4142) {$2$};
\node[nodo] (3) at (2,0) {$3$};
\node[nodo] (4) at (1.4142,-1.4142) {$4$};
\node[nodo] (5) at (0,-2) {$5$};
\node[nodo] (6) at (-1.4142,-1.4142) {$6$};
\node[nodo] (7) at (-2,0) {$7$};
\node[nodo] (8) at (-1.4142,1.4142) {$8$};

\draw (1) to (2);
\draw (2) to (3);
\draw (3) to (4);
\draw (4) to (5);
\draw (5) to (6);
\draw (6) to (7);
\draw (7) to (8);
\draw (8) to (1);
\draw (1) to (6);
\draw (2) to (5);
\draw (3) to (8);
\draw (4) to (7);
\end{tikzpicture} %
\\
$Q$
\end{tabular}
}
\end{center}
\caption{The Wagner graph $W$ and the cube $Q$. \label{figure:cubicosA}} 
\end{figure}
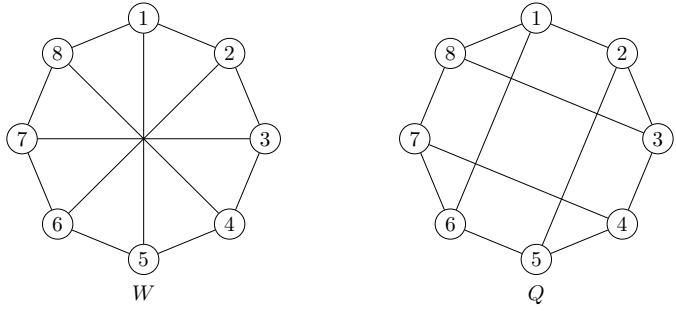

Let $G$ be a graph. An \emph{edge-cut} of $G$ is a set $F$ of edges of $G$ such that $G-F$ is disconnected. A \emph{$k$-edge-cut} is an edge-cut of size $k$. The \emph{edge-connectivity} of a graph $G\neq K_1$, denoted by $\lambda{(G)}$, is the minimum $k$ such that $G$ has a $k$-edge-cut. 

Let $F$ be an edge-cut of $G$. We say $F$ \emph{separates} a set $S$ of vertices $G$ if $F$ contains all edges with precisely one endpoint in $S$ and no edge with both endpoints in $S$. Notice that $F$ may also contain some edges with no endpoints in $S$. If $S=\{v\}$ for some vertex $v$ of $G$, we say that \emph{$F$ separates $v$} and that $F$ is \emph{vertex-separating}. If $S=\{u,v\}$ where $u$ and $v$ are the endpoints of an edge $e$ of $G$, we say that \emph{$F$ separates $e$} and that $F$ is \emph{edge-separating}. An edge-cut is \emph{nontrivial} if it is neither vertex-separating nor edge-separating. If $S$ induces a graph $H$ in $G$, we say $F$ is \emph{$H$-separating}.

Let $G$ be a graph with no loops. An edge $e$ is \emph{incident} to a vertex $v$ if $v$ is an endpoint of $e$. The \emph{degree} $d_G(v)$ of a vertex $v$ of $G$ is the number of edges incident to it. The minimum degree among the vertices of $G$ is denoted by $\delta(G)$. We say $G$ is \emph{cubic} if all its vertices have degree $3$. Two edges are \emph{nonincident}, \emph{incident}, or \emph{parallel} if they share precisely 0, 1, or 2 endpoints, respectively. A \emph{matching} is a set of pairwise nonincident edges. A matching $M$ of $G$ is \emph{perfect} if every vertex of $G$ is an endpoint of some edge in $M$. 

By \emph{subdividing $k$ times} an edge with endpoints $x$ and $y$, we mean replacing the edge by $k+1$ edges $xz_1,z_1z_2,\ldots,z_{k-1}z_k,z_ky$, where $z_1,z_2,\ldots,z_k$ are $k$ new vertices of degree $2$ each.

A simple graph $G$ is \emph{2-connected} if it has at least $3$ vertices, it is connected, and $G-v$ is connected for all $v$ in $V(G)$. Let $G$ be a $2$-connected simple graph having more edges than vertices. A \emph{chain $\gamma$ of $G$} is the edge set of a path $P$ in $G$, where all internal vertices of $P$ (if any) have degree $2$ in $G$ and $P$ has two distinct endpoints of degree greater than $2$ in $G$ each. The \emph{endpoints} of $\gamma$ are those of $P$ and $\gamma$ is \emph{incident} to a vertex $v$ if $v$ is one of its endpoints. The \emph{internal vertices} of $\gamma$ are those of $P$. By \emph{removing} $\gamma$ from $G$, we mean removing the edges and internal vertices of $\gamma$ (but not its endpoints). The graph that results by removing $\gamma$ from $G$ is denoted by $G\ominus\gamma$. If $H$ is a set of chains of $G$, we denote by $G\ominus H$ the graph that arises from $G$ by removing all the chains in $H$. By \emph{collapsing $\gamma$} we mean removing $\gamma$ and adding an edge with the same endpoints as $\gamma$. The \emph{length} of $\gamma$, denoted $\ell(\gamma)$, is the size $\vert\gamma\vert$. We denote by $\Gamma(G)$ the set of all chains of $G$. The \emph{distillation of $G$}, denoted $D(G)$, is the graph that arises from $G$ by collapsing all of its chains. Clearly, $G$ arises from $D(G)$ by a sequence of (possibly zero) subdivisions. Notice that every edge of $G$ belongs to precisely one chain and that $D(G)$ may have parallel edges but no loops (recall we are assuming $G$ is $2$-connected, simple, and has more edges than vertices). Two chains are \emph{nonincident}, \emph{incident}, or \emph{parallel} if they share precisely $0$, $1$, or $2$ endpoints, respectively. A \emph{matching of chains of $G$} is a set of pairwise nonincident chains of $G$. A \emph{perfect matching of chains} of $G$ is a matching of chains of $G$ whose endpoints are precisely all the vertices of $G$ having degree greater than $2$ in $G$. 

\section{Background}\label{section:background}

In this section, we present some previously known results regarding uniformly most reliable graphs. Along with this, some related concepts are discussed.

A simple graph $G$ is \emph{more reliable near $\rho=0$} (respectively, \emph{near $\rho=1$}) than a simple graph $H$ if there exists $\epsilon>0$ such that $R_G(\rho) > R_H(\rho)$ for all $\rho \in (0,\epsilon)$ (respectively, for all $\rho\in (1-\epsilon,1)$). A simple graph $G$ is \emph{locally most reliable near $\rho=0$} (respectively, $\rho=1$) if it is more reliable near $\rho=0$ (respectively, $\rho=1$) than all other simple graphs on the same number of vertices and edges as $G$. Notice that if there exists a locally most reliable graph near $\rho=0$ (respectively, $\rho=1$) on $n$ vertices and $m$ edges, then it is necessarily unique.

Let $G$ be a simple graph on $n$ vertices and $m$ edges. We denote by $\mu_k(G)$ the number of $k$-edge-cuts in $G$. Clearly, for each $\rho \in [0,1]$,
\[ R_G(\rho) = 1-\sum_{k=0}^{m}\mu_k(G)\rho^k(1-\rho)^{m-k}. \]
By comparing the values of the above polynomial expression for two different graphs, as $\rho$ approaches $0$ or $1$, the following can be proved.
\begin{theorem}[Brown and Cox~\cite{2014-Brown}]\label{theorem:optlocal}
Let $G$ and $H$ be simple graphs on $n$ vertices and $m$ edges.
\begin{enumerate}[label=(\roman*)]
    \item\label{it1:optlocal} If there exists $i\in \{0,1,\ldots,m\}$ such that $\mu_{k}(G)=\mu_k(H)$ for all $k\in\{0,1,\ldots,i-1\}$ and $\mu_i(G)<\mu_i(H)$, then $G$ is more reliable than $H$ near $\rho=0$.
    \item\label{it2:optlocal} If there exists $j\in \{0,1,\ldots,m\}$ such that $\mu_k(G)=m_k(H)$ for all $k\in\{j+1,j+2,\ldots,m\}$ and $\mu_j(G)<\mu_j(H)$, then $G$ is more reliable than $H$ near $\rho=1$.
\end{enumerate}
\end{theorem}

Let $G$ be a simple graph on $n$ vertices and $m$ edges. If $k\in\{0,1,\ldots,m\}$, we say $G$ is \emph{min-$\mu_k$} if $\mu_k(G) \leq \mu_k(H)$ for all other simple graphs $H$ on the same number of vertices and edges as $G$.

\begin{theorem}[Wang \cite{1994-Wang}]\label{theorem:strong}
Let $G$ be a simple graph on $n$ vertices and $m$ edges such that $m>n$. If $G$ is min-$\mu_k$ for some $k\in \{\lambda(G),\lambda(G)+1,\ldots,m-n+1\}$, then $G$ is 2-connected.
\end{theorem}

\begin{theorem}[Bauer et al.\ \cite{1985-Bauer}]\label{theorem:Bauer}
Let $G$ be a $2$-connected simple graph on $n$ vertices and $m$ edges such that $n+2\leq m \leq 3n/2$. Let $c$ be the corank of $G$ and let $r$ and $s$ be the unique integers such that $m=(3c-3)s+r$ and $r\in \{0,1,\ldots,3c-4\}$. Then, $G$ is min-$\mu_2$ if and only if the following two assertions hold: 
\begin{enumerate}[label=(\roman*)]
\item $D(G)$ is a simple, cubic, has $2c-2$ vertices, and $\lambda(D(G))=3$;
\item $G$ has $r$ chains of length $s+1$ and $3c-3-r$ chains of length $s$.
\end{enumerate}
\end{theorem}

\begin{theorem}[Wang \cite{1997-Wang}]\label{theorem:m3}
Let $G$ be a $2$-connected simple graph on $n$ vertices and $m$ edges such that $n+2\leq m\leq 3n/2$. Let $c$ be the corank of $G$ and let $r$ and $s$ be the unique integers such that $m=(3c-3)s+r$ and $r\in\{0,1,\ldots,3c-4\}$. 
Then, $G$ is min-$\mu_3$ if and only if all the following assertions hold:
\begin{enumerate}[label=(\roman*)]
    \item $D(G)$ is simple, cubic, has $2c-2$ vertices, is min-$\mu_3$, and $\lambda(D(G))=3$;
    \item $G$ has $r$ chains of length $s+1$ and $3c-3-r$ chains of length $s$;
    \item if $r\leq c-1$ (respectively, $r\geq 2c-2$), then the chains of $G$ of length $s+1$ (respectively, $s$) in $G$ form a matching of chains; whereas, if $c-1<r<2c-2$, then, for every three chains of $G$ having a common endpoint, there is at least one of them of length $s$ and at least one of them of length $s+1$.
\end{enumerate}
\end{theorem}

\begin{corollary}[Wang \cite{1997-Wang}]\label{cor:wang}
Let $G$ be a $2$-connected simple graph on $n$ vertices and $m$ edges such that $n+2\leq m<3n/2$. If $G$ is min-$\mu_3$, then $G$ is min-$\mu_2$.
\end{corollary}

Apart from a few classes $\mathcal C_{n,m}$ with $n\leq 8$ containing no UMRG found by Myrvold~\cite{1990-Myrvold} by exhaustive computer search, the following are all the previously known results on the nonexistence of UMRGs in $\mathcal C_{n,m}$. Notice that all these results are focused on classes $\mathcal C_{m,n}$ of dense graphs.

\begin{theorem}[Kelmans \cite{1981-Kelmans}; Myrvold et al.\ \cite{1991-Myrvold}]
There is no UMRG in the class $\mathcal C_{n,m}$ if any of the following assertions holds:
\begin{enumerate}[label=(\roman*)]
\item $n\geq 6$, $n$ is even, and $m=\binom{n}{2}-\frac{n+2}{2}$;
\item $n\geq 7$, $n$ is odd, and $m=\binom{n}{2}-\frac{n+5}{2}$.
\end{enumerate}
\end{theorem}

\begin{theorem}[Brown and Cox~\cite{2014-Brown}]
Let $n\geq 6$ and $m=\binom{n}{2}-(n-k)$, 
where $1 \leq k < n/2$. Then, there is no UMRG in the class $\mathcal C_{n,m}$ if any of the following assertions holds:
\begin{enumerate}[label=(\roman*)]
    \item $k=1$ and $n\equiv 1\pmod 3$;
    \item $n-2k \equiv 1\pmod 3$, $k\neq 1$, and $k \neq (n-1)/2$;
    \item $n-2k \equiv 2\pmod 3$ and $k\neq 1$.
\end{enumerate}
\end{theorem}

\section{Preliminaries}\label{section:preliminary}

In this section, we distinguish a class of 2-connected simple graphs having more edges than vertices that we call \emph{fair graphs}, characterized by the fact that the lengths of every two of its chains differ by at most one. The main result of this section is Proposition~\ref{prop:mu^I}, which states that, if $\mathcal S$ is a set of fair graphs having all the same number of vertices, edges, and chains, then, for suitable values of $k$, the problem of minimizing the number of $k$-edge-cuts over $\mathcal S$ is equivalent to the problem of minimizing the number of induced $k$-edge-cuts (see Definition~\ref{def:induced}) over $\mathcal S$.

Let $G$ be a $2$-connected simple graph having more edges than vertices. Recall that we denote by $\Gamma(G)$ the set of all chains of $G$. Moreover, if $k$ is a nonnegative integer, we denote by $\Gamma^{(k)}(G)$ the family of all subsets of $\Gamma(G)$ of size $k$; i.e., $\Gamma^{(k)}(G)=\binom{\Gamma(G)}k$. We also let
\[ \Gamma^{(k)}_-(G)=\{H\in\Gamma^{(k)}(G):\,G\ominus H\text{ is disconnected}\}. \]

The following lemma gives an expression for the number of $k$-edge-cuts of any $2$-connected simple graph.

\begin{lemma}\label{lemma:objetivo} 
For each $2$-connected simple graph $G$ on $n$ vertices and 
$m$ edges such that $m> n$ and each $k\in \{0,1,\ldots,m\}$,
\begin{equation}\label{eq:mk}
  \mu_k(G)=\binom mk-\sum_{H\in\Gamma^{(k)}(G)}\prod_{\gamma\in H}\ell(\gamma)+\sum_{H\in\Gamma^{(k)}_-(G)}\prod_{\gamma\in H}\ell(\gamma).
\end{equation}
\end{lemma}
\begin{proof} Notice that $\mu_k(G)$ can be computed by subtracting from $\binom mk$ the number of choices for $k$ edges of $G$ whose removal keeps $G$ connected. In order to prevent $G$ from becoming disconnected, each such choice of edges must consist of at most one edge from each chain and the edges must be taken from a set $H$ of chains so that $G\ominus H$ is not disconnected. Hence,
\begin{align*}
     \mu_k(G)&=\binom mk-\sum_{H\in\Gamma^{(k)}(G){}-\Gamma^{(k)}_-(G)}\prod_{\gamma\in H}\ell(\gamma)
          =\binom mk-\left(\sum_{H\in\Gamma^{(k)}(G)}\prod_{\gamma\in H } \ell(\gamma)-\sum_{H\in\Gamma^{(k)}_-(G)}\prod_{\gamma\in H} \ell(\gamma)\right). \qedhere
\end{align*}
\end{proof}

Notice that if the graph $G$ in the above lemma has precisely $t$ chains and the lengths of all its chains are $\ell_1,\ell_2,\ldots,\ell_t$, then the second term of the right-hand side of~\eqref{eq:mk} is
\begin{equation}\label{eq:aux} \sum_{H\in\Gamma^{(k)}}\prod_{\gamma\in H}\ell(\gamma)=\sum_{J\in\binom{[t]}{k}}\prod_{i\in J}\ell_i. \end{equation}
If $\ell_1+\ell_2+\cdots+\ell_t$ is constrained to be equal to a fixed value $m$, the right-hand side of the above equation is maximized when the tuple $(\ell_1,\ell_2,\ldots,\ell_t)$ is fair as defined below. This maximality result is proved in Lemma~\ref{lemma:fairness}.

\begin{definition}
A tuple $(x_1,x_2,\ldots,x_t)\in \mathbb{Z}_{+}^t$ is \emph{fair} if $|x_i-x_j|\leq 1$ for all $i,j \in \{1,2,\ldots,t\}$. A graph $G$ is \emph{fair} if it is a $2$-connected simple graph having more edges than vertices such that the tuple whose entries are the lengths of all the chains in $G$ is fair.
\end{definition}

\begin{lemma}\label{lemma:fairness} Let $k$ and $t$ be integers such that $2 \leq k\leq t$ and let
\[ \phi^{(k)}_{t}(\ell_1,\ell_2,\ldots,\ell_t)=\sum_{J\in\binom{[t]}{k}}\prod_{i\in J}\ell_i\qquad\text{for each }(\ell_1,\ell_2,\ldots,\ell_t)\in\mathbb Z_+^t. \]
Let $m$ be any integer such that $m \geq t$ and let $L_{t,m}=\{(\ell_1,\ell_2,\ldots,\ell_t)\in\mathbb Z_+^t:\ell_1+\ell_2+\cdots+\ell_t=m\}$. 
The maximum of $\phi_t^{(k)}(\ell_1,\ell_2,\ldots,\ell_t)$ as $(\ell_1,\ell_2,\ldots,\ell_t)$ ranges over $L_{t,m}$ is attained precisely at those tuples that are fair.
\end{lemma}

\begin{notation} The maximum attained by $\phi_t^{(k)}(\ell_1,\ell_2,\ldots,\ell_t)$ in $L_{t,m}$ will be denoted by $\Phi^{(k)}_{t}(m)$.\end{notation}

\begin{proof}[Proof of Lemma~\ref{lemma:fairness}]
Since $\phi^{(k)}_{t}(\ell_1,\ell_2,\ldots,\ell_t)$ is a symmetric polynomial on $\ell_1,\ell_2,\ldots,\ell_t$, we assume, without loss of generality, that $\ell_1\geq\ell_2\geq\cdots\geq \ell_t$. Thus, $(\ell_1,\ell_2,\ldots,\ell_t)$ is fair if and only if $\ell_1\leq\ell_t+1$. Hence, it suffices to prove that if $\ell_1 \geq \ell_t+2$, then $\phi_{t}^{(k)}(\ell_1,\ell_2,\ldots,\ell_t)<\phi_{t}^{(k)}(\ell_1-1,\ell_2,\ldots,\ell_{t-1},\ell_t+1)$. In fact,
\begin{align*}
&\phi^{(k)}_{t}(\ell_1-1,\ell_2,\ldots,\ell_{t-1},\ell_t+1)-\phi^{(k)}_{t}(\ell_1,\ell_2,\ldots,\ell_t)\\
&\quad=\sum_{J\in \binom{[t]-\{1,t\}}{k-2}}(\ell_1-1)(\ell_{t}+1)\prod_{i\in J}\ell_i+{}\\
&\quad\phantom{{}={}}\sum_{J\in \binom{[t]-\{1,t\}}{k-1}}(\ell_1-1)\prod_{i\in J}\ell_i+\sum_{J\in \binom{[t]-\{1,t\}}{k-1}}(\ell_t+1)\prod_{i\in J}\ell_i-{}\\
&\quad\phantom{{}={}}\left(\sum_{J\in \binom{[t]-\{1,t\}\}}{k-2}}\ell_1\ell_t\prod_{i\in J}\ell_i+\sum_{J\in \binom{[t]-\{1,t\}}{k-1}}\ell_1\prod_{i\in J}\ell_i+\sum_{J\in \binom{[t]-\{1,t\}}{k-1}}\ell_t\prod_{i\in J}\ell_i\right)\\
&\quad=\sum_{J\in \binom{[t]-\{1,t\}}{k-2}}(\ell_1-\ell_t-1)\prod_{i\in J}\ell_i>0.
\end{align*}
This proves $\phi^{(k)}_t$ can only attain its maximum over $L_{t,m}$ in tuples that are fair. As the value of $\phi^{(k)}_t$ is the same over all fair tuples in $L_{t,m}$, the proof the lemma is complete.
\end{proof}

\begin{definition}\label{def:induced}
Let $G$ be a $2$-connected simple graph having more edges than vertices. Let $\{f_1,f_2,\ldots,f_k\}$ be a $k$-edge-cut of $D(G)$. For each $i \in \{1,2,\ldots,k\}$, let $\gamma_i$ be the chain of $G$ corresponding to the edge $f_i$ of $D(G)$. We say a $k$-edge-cut $\{e_1,e_2,\ldots,e_k\}$ of $G$ \emph{is induced by $\{f_1,f_2,\ldots,f_k\}$} if $e_i \in \gamma_i$ for each $i \in \{1,2,\ldots,k\}$. Moreover,
\begin{enumerate}[label=(\roman*)]
\item if $\{f_1,f_2,\ldots,f_k\}$ is vertex-separating, then $\{e_1,e_2,\ldots,e_k\}$ is called \emph{Type-V};
\item if $\{f_1,f_2,\ldots,f_k\}$ is edge-separating but not vertex-separating, then
$\{e_1,e_2,\ldots,e_k\}$ is called \emph{Type-E};
\item if $\{f_1,f_2,\ldots,f_k\}$ is nontrivial, then $\{e_1,e_2,\ldots,e_k\}$ is called \emph{Type-N}.
\end{enumerate}
The number of Type-V, Type-E, and Type-N $k$-edges-cuts of $G$ is denoted by $\mu_k^{\mathrm V}(G)$, $\mu_k^{\mathrm E}(G)$, and $\mu_k^{\mathrm N}(G)$, respectively. The total number of induced $k$-edge-cuts of $G$ is denoted by $\mu_k^{\mathrm I}(G)$; i.e.,
\begin{equation*}
   \mu_k^{\mathrm I}(G)=\mu_k^{\mathrm V}(G)+\mu_k^{\mathrm E}(G)+\mu_k^{\mathrm N}(G).
\end{equation*}
\end{definition}

Notice that, by the definition of $\mu_k^{\mathrm I}(G)$, it coincides with the third term of the right-hand side of~\eqref{eq:mk}; i.e.,
\[ \mu_k^{\mathrm I}(G)=\sum_{H\in\Gamma^{(k)}_-(G)}\prod_{\gamma\in H}\ell(\gamma). \]
This fact combined with Lemma~\ref{lemma:fairness} leads to the following result.

\begin{proposition}\label{prop:mu^I} Let $k$ and $t$ be positive integers such that $2\leq k\leq t$ and let $\mathcal S$ be a nonempty set of fair graphs on $n$ vertices and $m$ edges such that $m>n$ and having precisely $t$ chains. Then, as $G$ ranges over $\mathcal S$, the minimum of $\mu_k(G)$ is attained precisely in the same graphs $G$ where the minimum of $\mu_k^{\mathrm I}(G)$ is attained. \end{proposition}
\begin{proof} Since the $t$ chains of $G$ are nonempty and pairwise disjoint, $t\leq m$. Thus, $k\leq m$ and, as $G$ is fair, Lemma~\ref{lemma:objetivo}, equation~\eqref{eq:aux}, and Lemma~\ref{lemma:fairness} imply that 
\begin{equation}\label{eq:mk2}
\mu_k(G)=\binom mk-\Phi^{(k)}_{t}(m)+\sum_{H\in\Gamma^{(k)}_-(G)}\prod_{\gamma\in H}\ell(\gamma).
\end{equation}
As $m$, $k$, and $t$ are fixed, the first two terms of the right-hand side of \eqref{eq:mk2} are constant. Hence, the minimum of $\mu_k(G)$ is attained precisely when the third term of the right-hand side of \eqref{eq:mk2} is minimized. And, as noted above, this term coincides with $\mu_k^{\mathrm I}(G)$.\end{proof}

\section{Locally most reliable graph near zero}\label{section:lmrg}

In this section, we characterize, for each positive integer $s$, the graph that is locally most reliable near $\rho=0$ in the class $\mathcal{C}_{12s+4,12s+8}$. This section is organized as follows. In Subsection~\ref{subsection:minm3}, all the min-$\mu_3$ graphs in $\mathcal{C}_{12s+4,12s+8}$ are found. Among these graphs, the only graph that minimizes $\mu_4$ is identified in Subsection~\ref{subsection:minm4}. In Subsection~\ref{subsection:lmrg-final} we prove that this graph is locally most reliable near $\rho=0$.

\subsection{Minimization of \texorpdfstring{$\mu_3$}{mu3}}\label{subsection:minm3}

In this subsection, we find the list of all the min-$\mu_3$ graphs in $\mathcal{C}_{12s+4,12s+8}$ for each positive integer $s$. The following remarks will be useful in achieving this goal.

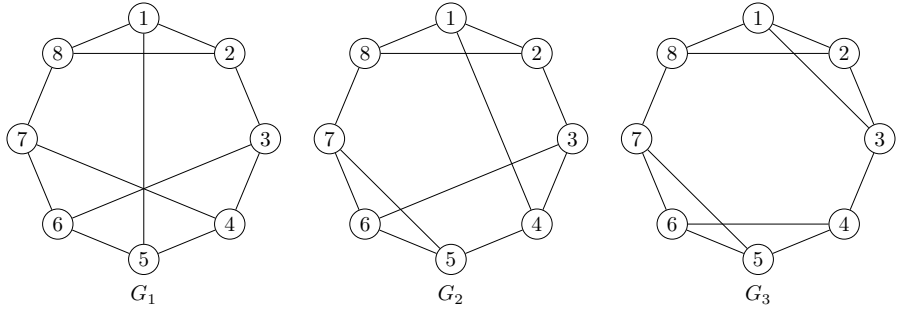
\begin{figure}
\begin{center}
\scalebox{0.8}{
\begin{tabular}{c}
\begin{tikzpicture}[ scale=1, nodo/.style={circle,draw=black!120,fill=white!120,inner sep=0pt,minimum size=5mm}]

\node[nodo] (1) at (0,2) {$1$};
\node[nodo] (2) at (1.4142,1.4142) {$2$};
\node[nodo] (3) at (2,0) {$3$};
\node[nodo] (4) at (1.4142,-1.4142) {$4$};
\node[nodo] (5) at (0,-2) {$5$};
\node[nodo] (6) at (-1.4142,-1.4142) {$6$};
\node[nodo] (7) at (-2,0) {$7$};
\node[nodo] (8) at (-1.4142,1.4142) {$8$};

\draw (1) to (2);
\draw (2) to (3);
\draw (3) to (4);
\draw (4) to (5);
\draw (5) to (6);
\draw (6) to (7);
\draw (7) to (8);
\draw (8) to (1);
\draw (1) to (5);
\draw (2) to (8);
\draw (3) to (6);
\draw (4) to (7);
\end{tikzpicture} %
\\
$G_1$
\end{tabular}
\begin{tabular}{c}
\begin{tikzpicture}[ scale=1, nodo/.style={circle,draw=black!120,fill=white!120,inner sep=0pt,minimum size=5mm}]

\node[nodo] (1) at (0,2) {$1$};
\node[nodo] (2) at (1.4142,1.4142) {$2$};
\node[nodo] (3) at (2,0) {$3$};
\node[nodo] (4) at (1.4142,-1.4142) {$4$};
\node[nodo] (5) at (0,-2) {$5$};
\node[nodo] (6) at (-1.4142,-1.4142) {$6$};
\node[nodo] (7) at (-2,0) {$7$};
\node[nodo] (8) at (-1.4142,1.4142) {$8$};

\draw (1) to (2);
\draw (2) to (3);
\draw (3) to (4);
\draw (4) to (5);
\draw (5) to (6);
\draw (6) to (7);
\draw (7) to (8);
\draw (8) to (1);
\draw (1) to (4);
\draw (2) to (8);
\draw (3) to (6);
\draw (5) to (7);
\end{tikzpicture} %
\\
$G_2$
\end{tabular}
\begin{tabular}{c}
\begin{tikzpicture}[ scale=1, nodo/.style={circle,draw=black!120,fill=white!120,inner sep=0pt,minimum size=5mm}]

\node[nodo] (1) at (0,2) {$1$};
\node[nodo] (2) at (1.4142,1.4142) {$2$};
\node[nodo] (3) at (2,0) {$3$};
\node[nodo] (4) at (1.4142,-1.4142) {$4$};
\node[nodo] (5) at (0,-2) {$5$};
\node[nodo] (6) at (-1.4142,-1.4142) {$6$};
\node[nodo] (7) at (-2,0) {$7$};
\node[nodo] (8) at (-1.4142,1.4142) {$8$};

\draw (1) to (2);
\draw (2) to (3);
\draw (3) to (4);
\draw (4) to (5);
\draw (5) to (6);
\draw (6) to (7);
\draw (7) to (8);
\draw (8) to (1);
\draw (1) to (3);
\draw (2) to (8);
\draw (4) to (6);
\draw (5) to (7);
\end{tikzpicture} %
\\
$G_3$
\end{tabular}
}
\end{center}
\caption{Three cubic graphs on 8 vertices. \label{figure:cubicosB}}
\end{figure}

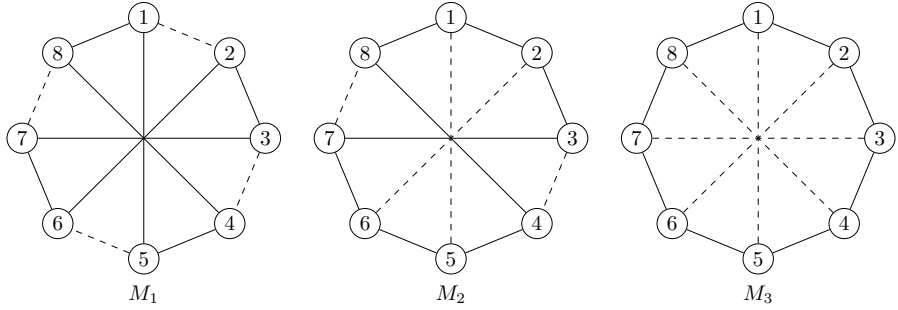
\begin{figure}
\begin{center}
\scalebox{0.8}{
\begin{tabular}{c}
\begin{tikzpicture}[ scale=1, nodo/.style={circle,draw=black!120,fill=white!120,inner sep=0pt,minimum size=5mm}]
     
\node[nodo] (1) at (0,2) {$1$};
\node[nodo] (2) at (1.4142,1.4142) {$2$};
\node[nodo] (3) at (2,0) {$3$};
\node[nodo] (4) at (1.4142,-1.4142) {$4$};
\node[nodo] (5) at (0,-2) {$5$};
\node[nodo] (6) at (-1.4142,-1.4142) {$6$};
\node[nodo] (7) at (-2,0) {$7$};
\node[nodo] (8) at (-1.4142,1.4142) {$8$};

\draw (2) to (3);
\draw (4) to (5);
\draw (6) to (7);
\draw (8) to (1);
\draw (1) to (5);
\draw (2) to (6);
\draw (3) to (7);
\draw (4) to (8);

\draw[dashed] (1) to (2);
\draw[dashed] (3) to (4);
\draw[dashed] (5) to (6);
\draw[dashed] (7) to (8);

\end{tikzpicture}
\\
$M_1$
\end{tabular}
\begin{tabular}{c}
\begin{tikzpicture}[ scale=1, nodo/.style={circle,draw=black!120,fill=white!120,inner sep=0pt,minimum size=5mm}]

\node[nodo] (1) at (0,2) {$1$};
\node[nodo] (2) at (1.4142,1.4142) {$2$};
\node[nodo] (3) at (2,0) {$3$};
\node[nodo] (4) at (1.4142,-1.4142) {$4$};
\node[nodo] (5) at (0,-2) {$5$};
\node[nodo] (6) at (-1.4142,-1.4142) {$6$};
\node[nodo] (7) at (-2,0) {$7$};
\node[nodo] (8) at (-1.4142,1.4142) {$8$};

\draw (1) to (2);
\draw (2) to (3);
\draw (4) to (5);
\draw (5) to (6);
\draw (6) to (7);
\draw (8) to (1);
\draw (4) to (8);
\draw (3) to (7);

\draw[dashed] (1) to (5);
\draw[dashed] (2) to (6);
\draw[dashed] (3) to (4);
\draw[dashed] (7) to (8);

\end{tikzpicture} %
\\
$M_2$
\end{tabular}
\begin{tabular}{c}
\begin{tikzpicture}[ scale=1, nodo/.style={circle,draw=black!120,fill=white!120,inner sep=0pt,minimum size=5mm}]

\node[nodo] (1) at (0,2) {$1$};
\node[nodo] (2) at (1.4142,1.4142) {$2$};
\node[nodo] (3) at (2,0) {$3$};
\node[nodo] (4) at (1.4142,-1.4142) {$4$};
\node[nodo] (5) at (0,-2) {$5$};
\node[nodo] (6) at (-1.4142,-1.4142) {$6$};
\node[nodo] (7) at (-2,0) {$7$};
\node[nodo] (8) at (-1.4142,1.4142) {$8$};

\draw[dashed] (1) to (5);
\draw[dashed] (2) to (6);
\draw[dashed] (3) to (7);
\draw[dashed] (4) to (8);

\draw (1) to (2);
\draw (2) to (3);
\draw (3) to (4);
\draw (4) to (5);
\draw (5) to (6);
\draw (6) to (7);
\draw (7) to (8);
\draw (8) to (1);

\end{tikzpicture} %
\\
$M_3$
\end{tabular}
}
\end{center}
\caption{Perfect matchings $M_1$, $M_2$, and $M_3$ of the Wagner graph $W$. The matchings consist of the dashed edges. 
\label{figure:W3}}
\end{figure}

\begin{figure}
\begin{center}
\scalebox{0.8}{
\begin{tabular}{c}
\begin{tikzpicture}[ scale=1, nodo/.style={circle,draw=black!120,fill=white!120,inner sep=0pt,minimum size=5mm}]

\node[nodo] (1) at (0,2) {$1$};
\node[nodo] (2) at (1.4142,1.4142) {$2$};
\node[nodo] (3) at (2,0) {$3$};
\node[nodo] (4) at (1.4142,-1.4142) {$4$};
\node[nodo] (5) at (0,-2) {$5$};
\node[nodo] (6) at (-1.4142,-1.4142) {$6$};
\node[nodo] (7) at (-2,0) {$7$};
\node[nodo] (8) at (-1.4142,1.4142) {$8$};

\draw[dashed] (1) to (6);
\draw[dashed] (2) to (5);
\draw[dashed] (3) to (8);
\draw[dashed] (4) to (7);

\draw (1) to (2);
\draw (2) to (3);
\draw (3) to (4);
\draw (4) to (5);
\draw (5) to (6);
\draw (6) to (7);
\draw (7) to (8);
\draw (8) to (1);

\end{tikzpicture} %
\\
$M_4$
\end{tabular}
\qquad\quad
\begin{tabular}{c}
\begin{tikzpicture}[ scale=1, nodo/.style={circle,draw=black!120,fill=white!120,inner sep=0pt,minimum size=5mm}]

\node[nodo] (1) at (0,2) {$1$};
\node[nodo] (2) at (1.4142,1.4142) {$2$};
\node[nodo] (3) at (2,0) {$3$};
\node[nodo] (4) at (1.4142,-1.4142) {$4$};
\node[nodo] (5) at (0,-2) {$5$};
\node[nodo] (6) at (-1.4142,-1.4142) {$6$};
\node[nodo] (7) at (-2,0) {$7$};
\node[nodo] (8) at (-1.4142,1.4142) {$8$};

\draw[dashed] (2) to (3);
\draw[dashed] (4) to (5);
\draw[dashed] (6) to (7);
\draw[dashed] (8) to (1);

\draw (1) to (2);
\draw (3) to (4);
\draw (5) to (6);
\draw (7) to (8);
\draw (1) to (6);
\draw (2) to (5);
\draw (3) to (8);
\draw (4) to (7);
\end{tikzpicture} %
\\
$M_5$
\end{tabular}
}
\end{center}
\caption{Perfect matchings $M_4$ and $M_5$ of the cube $Q$. The matchings consist of the dashed edges.\label{figure:Q2}}
\end{figure}
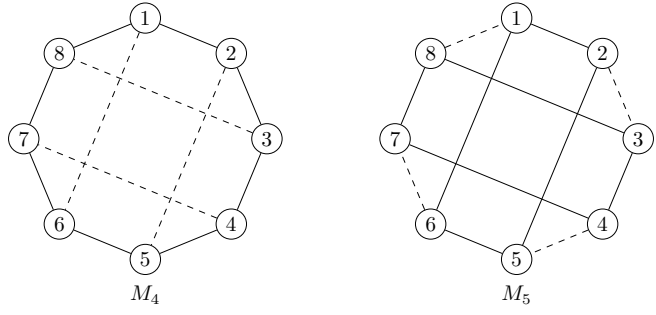

\begin{remark}\label{remark:Bussemake}
As first reported in~\cite{Bussemake1977}, there are, up to isomorphism, precisely five cubic graphs on 8 vertices. These five graphs are those depicted in Figures~\ref{figure:cubicosA}~and~\ref{figure:cubicosB}. On the one hand, both graphs in Figure~\ref{figure:cubicosA} (i.e., $W$ and $Q$) are min-$\mu_3$ because, as it is easy to verify, their only 3-edge-cuts are the $8$ vertex-separating ones. On the other hand, the graphs depicted in Figure~\ref{figure:cubicosB} are not min-$\mu_3$. In fact, each of $G_1$, $G_2$, and $G_3$ has, in addition to the $8$ vertex-separating $3$-edge-cuts, some other $3$-edge-cut; e.g., $\{23,78,15\}$, $\{23,78,14\}$, and $\{23,78,13\}$, respectively.
\end{remark}

\begin{remark}\label{remark:matchings}
The only perfect matchings of $W$, up to isomorphism, are the matchings $M_1$, $M_2$, and $M_3$ depicted in Figure~\ref{figure:W3}. The only perfect matchings of $Q$ up to isomorphism are the matchings $M_4$ and $M_5$ in Figure~\ref{figure:Q2}.
\end{remark}

As it will turn out (see Lemma~\ref{lemma:five-m3}), the min-$\mu_3$ graphs in $\mathcal C_{12s+4,12s+8}$ are certain enlarged graphs, a notion we define below.

\begin{definition}
Let $G$ be a cubic simple graph, $Y$ a set of edges of $G$, and $s$ a positive integer. The \emph{enlarged graph} $G_s^{Y}$ is the graph that arises from $G$ by subdividing $s-1$ times each edge in $Y$ and $s$ times each edge in $E(G)-Y$. 
\end{definition}
Notice that if, in addition, $G$ is 2-connected, then $G^Y_s$ arises from $G$ by replacing each edge in $Y$ by a chain of length $s$ and each edge not in $Y$ by one of length $s+1$.

We are in a position to give the list of all the min-$\mu_3$ graphs in $\mathcal{C}_{12s+4,12s+8}$. 

\begin{lemma}\label{lemma:five-m3}
Let $s$ be a positive integer. A graph $G$ in $\mathcal{C}_{12s+4,12s+8}$ is min-$\mu_3$ if and only if $G$ is one of the five enlarged graphs $W^{M_1}_{s}$, $W^{M_2}_{s}$, $W^{M_3}_{s}$, $Q^{M_4}_{s}$, and $Q^{M_5}_{s}$, where $M_1,\ldots,M_5$ are the perfect matchings depicted in Figures~\ref{figure:W3} and \ref{figure:Q2}.
\end{lemma}
\begin{proof}
Let $G$ be a min-$\mu_3$ graph in $\mathcal{C}_{n,m}$, where $n=12s+4$ and $m=12s+8$. Since $s\geq 1$, $m<3n/2$. As a consequence, the minimum degree of $G$ is at most $2$ and, in particular, $\lambda(G)\leq 2$. Since $m>n$, $G$ is min-$\mu_3$, and $3 \in \{\lambda(G),\lambda(G)+1,\ldots,m-n+1\}$, Theorem~\ref{theorem:strong} ensures that $G$ is 2-connected. Hence, as the corank of $G$ is $c=5$ and $m=12s+8$, Theorem~\ref{theorem:m3} ensures that all the following assertions hold:
\begin{enumerate}[label=(\roman*)]
    \item\label{it1:int} $D(G)$ is a min-$\mu_3$ simple cubic graph on $8$ vertices;
    \item $G$ has $8$ chains of length $s+1$ and $4$ chains of length $s$; \label{it2:int}
    \item\label{it3:int} the $4$ chains of length $s$ in $G$ form a perfect matching of chains of $G$.
\end{enumerate}
Assertion~\ref{it1:int} and Remark~\ref{remark:Bussemake} imply that $D(G)$ is either $W$ or $Q$. Hence, assertions~\ref{it2:int} and~\ref{it3:int} and Remark~\ref{remark:matchings} imply that $G$ is isomorphic to
$W^{M_1}_{s}$, $W^{M_2}_{s}$, $W^{M_3}_{s}$, $Q^{M_4}_{s}$, or $Q^{M_5}_{s}$.
\end{proof}

\subsection{Minimization of \texorpdfstring{$\mu_4$}{mu4}} \label{subsection:minm4}

The main result of this subsection is Lemma~\ref{lemma:m4} that states $W^{M_1}_{s}$ is the unique graph that minimizes $\mu_4$ among the min-$\mu_3$ graphs in $\mathcal{C}_{12s+4,12s+8}$. Our strategy for proving this is as follows. By Lemma~\ref{lemma:five-m3}, all min-$\mu_3$ graphs in $\mathcal C_{12s+4,12s+8}$ are fair. Hence, Proposition~\ref{prop:mu^I} implies that, to establish Lemma~\ref{lemma:m4}, it suffices to prove that $W_s^{M_1}$ minimizes $\mu_4^{\mathrm I}$ over all min-$\mu_3$ graphs in $\mathcal C_{12s+4,12s+8}$. Lemmas~\ref{lemma:trivial-4-cuts-v} and~\ref{lemma:trivial-4-cuts-e} show that $\mu_4^{\mathrm V}(G)$ and $\mu_4^{\mathrm E}(G)$ are constant for all min-$\mu_3$ graphs $G$ in $\mathcal C_{12s+4,12s+8}$. Lemma~\ref{lemma:nontrivial-4-cuts} further shows that, among these same graphs, $\mu_4^{\mathrm N}(G)$ is minimized only when $G=W_s^{M_1}$. As $\mu_4^I(G)=\mu_4^{\mathrm V}(G)+\mu_4^{\mathrm E}(G)+\mu_4^{\mathrm N}(G)$, the main result of this subsection follows.

\begin{lemma}\label{lemma:trivial-4-cuts-v}
If $G \in \mathcal{C}_{12s+4,12s+8}$ and $G$ is min-$\mu_3$, then 
\begin{equation*}
  \mu_4^{\mathrm V}(G)=8s(s+1)^2(9s+6).
\end{equation*} 
\end{lemma}
\begin{proof}
By Lemma \ref{lemma:five-m3}, $G$ is $W^{M_1}_{s}$, $W^{M_2}_{s}$, $W^{M_3}_{s}$, $Q^{M_4}_{s}$, or $Q^{M_5}_{s}$. In particular, $D(G)$ is either $W$ or $Q$. Let $v$ be a vertex of $D(G)$. As $D(G)$ is cubic, the $4$-edge-cuts of $D(G)$ separating $v$ are those consisting precisely of the three edges incident to $v$ plus some edge nonincident to $v$. Thus, the $4$-edge-cuts of $G$ induced by some $4$-edge-cut of $D(G)$ separating $v$ are those consisting precisely of an edge from each chain incident to $v$ plus an edge from a chain nonincident to $v$ in $G$. Since $G$ is an enlarged graph of a cubic graph by a perfect matching, $v$ is incident precisely to $1$ chain of length $s$ and $2$ chains of length $s+1$. Thus, the number of $4$-edge-cuts of $G$ induced by a $4$-edge-cut of $D(G)$ that separates $v$ equals $s(s+1)^2(12s+8-(s+2(s+1)))=s(s+1)^2(9s+6)$. As $D(G)$ has precisely $8$ vertices and no $4$-edge-cut of $D(G)$ separates more than one vertex (because $D(G)$ is cubic), the lemma follows.\end{proof}

\begin{lemma}\label{lemma:trivial-4-cuts-e}
If $G \in \mathcal{C}_{12s+4,12s+8}$ and $G$ is min-$\mu_3$, then 
\begin{equation*}
  \mu_4^{\mathrm E}(G)=4(s+1)^4 + 8s^2(s+1)^2.
\end{equation*} 
\end{lemma}
\begin{proof}
By Lemma \ref{lemma:five-m3}, $G$ is $W^{M_1}_{s}$, $W^{M_2}_{s}$, $W^{M_3}_{s}$, $Q^{M_4}_{s}$, or $Q^{M_5}_{s}$ and let $M$ be $M_1$, $M_2$, $M_3$, $M_4$, or $M_5$, respectively. In particular, $D(G)$ is either $W$ or $Q$. Let $e$ be an edge of $D(G)$. As $D(G)$ is cubic, the only $4$-edge-cut of $D(G)$ separating $e$ is the set $F$ consisting precisely of the four edges of $D(G)$ incident to $e$. Let $\gamma$ be the chain of $G$ corresponding to $e$. Thus, the $4$-edge-cuts of $G$ induced by $F$ are those consisting precisely of one edge of each chain of $G$ incident to $\gamma$. On the one hand, if $e\in M$, then the four chains incident to $\gamma$ in $G$ have length $s+1$, which means that the number of $4$-edge-cuts of $G$ induced by $F$ is $(s+1)^4$. On the other hand, if $e\notin M$, then, since $M$ is a perfect matching, two of the chains of $G$ incident to $\gamma$ have length $s$ and the other two have length $s+1$, implying that the number of $4$-edge-cuts of $G$ induced by $F$ is $s^2(s+1)^2$. Since $G$ has $4$ edges in $M$ and $8$ edges not in $M$ and no $4$-edge-cut of $D(G)$ separates more than one edge, the lemma follows.
\end{proof}

\begin{lemma}\label{lemma:nontrivial-4-cuts}
Let $s$ be a positive integer. 
The graph $W_s^{M_1}$ is the only graph that minimizes $\mu_4^{\mathrm N}$ among all the min-$\mu_3$ graphs in $\mathcal C_{12s+4,12s+8}$.
\end{lemma}
\begin{proof}
Let $G$ be a min-$\mu_3$ graph in $\mathcal C_{12s+4,12s+8}$. By Lemma~\ref{lemma:five-m3}, $G$ is $W^{M_1}_{s}$, $W^{M_2}_{s}$, $W^{M_3}_{s}$, $Q^{M_4}_{s}$, or $Q^{M_5}_{s}$.

Suppose first that $D(G)=W$. Observe that $W$ has only two nontrivial 4-edge-cuts, namely, $M_1$ and $M_1'= \{23,45,67,81\}$. The Type-N $4$-edge-cuts of $G$ are those induced by $M_1$ or $M_1'$. On the one hand, all the chains of $G$ corresponding to edges of $M_1$ have length at least $s$ and they all have length $s$ only when $G=W_s^{M_1}$. On the other hand, in all the cases (i.e., $G$ equals $W_s^{M_1}$, $W_s^{M_2}$, or $W_s^{M_3}$), the four chains of $G$ corresponding to edges in $M_1'$ are of length $s+1$. We conclude that $\mu_4^{\mathrm N}(G)\geq s^4+(s+1)^4$, with equality only when $G=W_s^{M_1}$.

Suppose now that $D(G)=Q$. Notice that the following two are nontrivial 4-edge-cuts of $Q$: $F = \{12,56,38,47\}$ and $M_1$. Notice that if $G$ is either $Q_s^{M_4}$ or $G=Q_s^{M_5}$, then at least two chains of $G$ corresponding to edges of $F$ are of length $s+1$ and all the chains of $G$ corresponding to edges in $M_1$ are of length $s+1$. Thus, in all cases (i.e., $G$ equals $Q_s^{M_4}$ or $Q_s^{M_5}$), $\mu_4^{\mathrm N}(G)\geq s^2(s+1)^2+(s+1)^4>\mu_4^{\mathrm N}(W_s^{M_1})$. This completes the proof of the lemma.
\end{proof}

\begin{lemma}\label{lemma:m4} For each positive integer $s$, $W^{M_1}_{s}$ is the only graph that minimizes the number of 4-edge-cuts among all min-$\mu_3$ graphs in 
$\mathcal{C}_{12s+4,12s+8}$.
\end{lemma}
\begin{proof}
Let $\mathcal S$ be the set of min-$\mu_3$ graphs in $\mathcal C_{12s+4,12s+8}$. By Lemma~\ref{lemma:five-m3}, $\mathcal S$ is a set of fair graphs having precisely $12$ chains each. Thus, by Proposition~\ref{prop:mu^I}, $\mu_4$ attains its minimum over $\mathcal S$ in the same graphs where $\mu_4^{\mathrm I}$ attains its minimum over $\mathcal S$. By Lemmas~\ref{lemma:trivial-4-cuts-v}~and~\ref{lemma:trivial-4-cuts-e}, the values $\mu_4^{\mathrm V}$ and $\mu_4^{\mathrm E}$ are constant over $\mathcal S$. Hence, as $\mu_4^{\mathrm I}=\mu_4^{\mathrm V}+\mu_4^{\mathrm E}+\mu_4^{\mathrm N}$, the result follows by Lemma~\ref{lemma:nontrivial-4-cuts}.
\end{proof}

\subsection{Local optimality near zero}\label{subsection:lmrg-final}
By combining Lemmas~\ref{lemma:five-m3} and~\ref{lemma:m4} with Theorem~\ref{theorem:optlocal} and Corollary~\ref{cor:wang}, we are able characterize the graph in $\mathcal C_{12s+4,12s+8}$ that is locally most reliable near $\rho=0$.

\begin{proposition}\label{proposition:rho0}
For each positive integer $s$, the graph $W^{M_1}_{s}$ is locally most reliable near $\rho=0$.
\end{proposition}
\begin{proof}
Clearly, $W_s^{M_1}$ is $2$-edge-connected; i.e., $\mu_0(W^{M_1}_{s})=\mu_1(W^{M_1}_{s})=0$. Thus, $W^{M_1}_{s}$ is min-$\mu_k$ when $k$ equals $0$ or $1$. By Lemma~\ref{lemma:five-m3}, $W_s^{M_1}$ is min-$\mu_3$. Moreover, reasoning as the proof of Lemma~\ref{lemma:five-m3}, $W_s^{M_1}$ satisfies the assumptions of Corollary~\ref{cor:wang} and, consequently, $W_s^{M_1}$ is also min-$\mu_2$. Furthermore, Lemma~\ref{lemma:m4} ensures that $W_s^{M_1}$ is the only graph that minimizes $\mu_4$ among the min-$\mu_3$ graphs. This proves that for any simple graph $H$ on $12s+4$ vertices and $12s+8$ edges different from $W^{M_1}_s$, there is some $i\in\{0,1,2,3,4\}$, such that $\mu_k(W^{M_1}_s)=\mu_k(H)$ for every $k\in\{0,1,\ldots,i-1\}$ and $\mu_i(W^{M_1}_s)<\mu_i(H)$. Therefore, the proposition follows by Theorem~\ref{theorem:optlocal}\ref{it1:optlocal}.
\end{proof}

\section{Proof of the main result}\label{section:main}

In this section we prove our main result, namely, Theorem~\ref{theorem:principal}, which asserts that there is no UMRG in $\mathcal C_{12s+4,12+8}$, for each positive integer $s$. For the purpose of establishing this result, we let $X=\{23,56,78,18\}$ (see Figure~\ref{figure:Wn}) and will show that the enlarged graph $W^X_s$ has fewer $5$-edge-cuts than $W^{M_1}_s$. By Proposition~\ref{prop:mu^I}, proving that $\mu_5({W^X_s})<\mu_5(W_s^{M_1})$ is equivalent to proving that $\mu_5^{\mathrm I}(W^X_s)<\mu_5^{\mathrm I}(W_s^{M_1})$. Before getting into the computation of the number of induced $5$-edge-cuts of $W^X_s$ and $W_s^{M_1}$, we prove a useful property of the $5$-edge-cuts of its common distillation $W$.

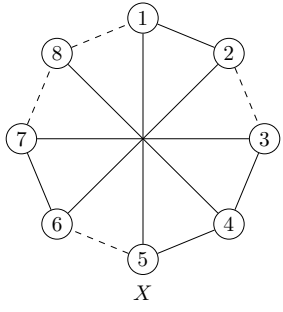
\begin{figure}
\begin{center}
\scalebox{0.8}{
\begin{tabular}{c}
\begin{tikzpicture}[ scale=1, nodo/.style={circle,draw=black!120,fill=white!120,inner sep=0pt,minimum size=5mm}]

\node[nodo] (1) at (0,2) {$1$};
\node[nodo] (2) at (1.4142,1.4142) {$2$};
\node[nodo] (3) at (2,0) {$3$};
\node[nodo] (4) at (1.4142,-1.4142) {$4$};
\node[nodo] (5) at (0,-2) {$5$};
\node[nodo] (6) at (-1.4142,-1.4142) {$6$};
\node[nodo] (7) at (-2,0) {$7$};
\node[nodo] (8) at (-1.4142,1.4142) {$8$};

\draw (1) to (2);
\draw (4) to (5);
\draw (6) to (7);
\draw (3) to (4);
\draw (1) to (5);
\draw (2) to (6);
\draw (3) to (7);
\draw (4) to (8);

\draw[dashed] (2) to (3);
\draw[dashed] (7) to (8);
\draw[dashed] (5) to (6);
\draw[dashed] (8) to (1);
\end{tikzpicture}
\\
$X$
\end{tabular}
}
\end{center}
\caption{The set $X$ consists of the dashed edges. \label{figure:Wn}}
\end{figure}

\begin{lemma}\label{lemma:count5}
Each nontrivial 5-edge-cut of $W$ is either $P_3$-separating or is $C_4$-separating.
\end{lemma}
\begin{proof} 
We begin the proof by enumerating four families of $5$-edge-cuts of $W$. Then, we will argue that this list is exhaustive.
\begin{itemize}
    \item \emph{Vertex-separating cuts}. Let $v$ be a vertex of $W$. Since $W$ is a cubic graph, each 5-edge-cut that separates $v$ consists of the three edges incident to $v$ plus two additional edges. Thus, the number of $5$-edge-cuts of $W$ that separate $v$ equals $\binom 92$ (the number of pairs of edges nonincident to $v$).
    Notice that the $5$-edge-cuts of $W$ that separate at least two distinct vertices are those separating two vertices that are adjacent and consist of the five edges that are incident to at least one of these two vertices. By the inclusion-exclusion principle, the number of vertex-separating $5$-edge-cuts of $W$ is $8 \times \binom{9}{2}-12=276$.
    \item \emph{Edge-separating cuts.} Let $e$ be an edge of $W$. The $5$-edge-cuts of $W$ that separate $e$ consist of the four edges incident to $e$ plus an additional edge different from $e$. Thus, the number of $5$-edge-cuts separating $e$ are $12-4-1=7$. As no $5$-edge-cut of $W$ separates two distinct edges, the total number of edge-separating $5$-edge-cuts of $W$ is $12\times 7=84$.
    
    \item \emph{$P_3$-separating cuts.} Let $S$ be a set of vertices of $W$ inducing $P_3$. As $W$ is cubic, the only $5$-edge-cut of $W$ that separate $S$ is the set $F$ consisting of the five edges having exactly one endpoint in $S$. Notice that conversely, $F$ determines $S$ univocally. Thus, the number of $P_3$-separating $5$-edges-cuts of $W$ equals the number of sets $S$ inducing $P_3$ in $W$. As $W$ is triangle-free, this equals the number of pairs of incident edges. As $W$ is a cubic graph on $8$ vertices, this number is $8\times 3=24$.
    
    \item \emph{$C_4$-separating cuts.} Observe that $W$ has four induced subgraphs isomorphic to $C_4$ whose vertex sets are $D_1=\{1,2,6,5\}$, $D_2=\{2,3,7,6\}$, $D_3=\{3,4,8,7\}$, or $D_4=\{4,5,1,8\}$. For each $i\in\{1,2,3,4\}$, the $5$-edge-cuts of $W$ that separate $D_i$ are those consisting of the four edges of $W$ incident to exactly one vertex of $D_i$ plus an additional edge incident to no vertex of $D_i$. Thus, the number of $D_i$-separating $5$-edge-cuts of $W$ is $4$. As no $5$-edge-cut of $W$ can separate $D_i$ and $D_j$ for $i\neq j$, the total number of $C_4$-separating $5$-edge-cuts of $W$ is $4\times 4=16$.
\end{itemize}
Notice that no $5$-edge-cut of $W$ belongs to more than one of the four families of edge-cuts discussed above. Thus, we have presented $400$ distinct $5$-edge-cuts of $W$. Notice that, if one removes five edges from $W$, the result is either a disconnected graph or a tree. Hence, $\mu_5(W)=\binom{12}{5}-t(W)$, where $t(W)$ denotes the number of spanning trees of $W$. Moreover, it is known~\cite{1993-Biggs} that $t(W)=392$. Therefore, $\mu_5(W)=\binom{12}{5}-392=400$. This means that the set of $5$-edge-cuts discussed in the four families above exhaust all $5$-edge-cuts of $W$. As only the last two families contain nontrivial edge-cuts, the lemma follows.\end{proof}

Let $G$ be either $W^{M_1}_{s}$ or $W^X_s$. Recall that $\mu_5^{\mathrm V}(G)$ and $\mu_5^{\mathrm E}(G)$ denote the number of Type-V and Type-E 5-edge-cuts of $G$, respectively. Further, we will denote by $\mu_5^{P_3}(G)$ 
(respectively, $\mu_5^{C_4}(G)$) the number 5-edge-cuts of $G$ induced by $P_3$-separating (respectively, $C_4$-separating) edge-cuts of $D(G)$. Thus, by the above lemma,
\begin{equation}\label{eq:muI}
  \mu_5^{\mathrm I}(G)=\mu_5^{\mathrm V}(G)+\mu_5^{\mathrm E}(G)+\mu_5^{P_3}(G)+\mu_5^{C_4}(G).
\end{equation}

In Lemmas~\ref{lemma:node5} to \ref{lemma:match5} below, we determine 
$\mu_5^{\mathrm V}(G)$, $\mu_5^{\mathrm E}(G)$, $\mu_5^{P_3}(G)$ and $\mu_5^{C_4}(G)$, respectively, for both $G=W^{M_1}_s$ and $G=W^X_s$.

\begin{lemma}\label{lemma:node5}
For each positive integer $s$,
\begin{align*}
   \mu_5^{\mathrm V}(W^{M_1}_{s})&=276s^5+920s^4+1128s^3+600s^2+116s\quad\qquad\text{and}\\
   \mu_5^{\mathrm V}(W^X_{s})&=276s^5+920s^4+1149s^3+651s^2+156s+10.
\end{align*}
\end{lemma}
\begin{proof}
Let us begin the proof by examining the $5$-edge-cuts $F$ of $W$ that separate a certain vertex $v$ in more detail. If follows from the proof of Lemma~\ref{lemma:count5} that these sets $F$ are precisely those consisting of the three edges incident to $v$ plus two additional edges nonincident to $v$. As among the edges of $W$ nonincident to $v$ there are $3$ edges in $M_1$ and $6$ edges not in $M_1$, of all possible $5$-edge-cuts of $W$ separating $v$ there are: (i) $\binom 32=3$ whose edges nonincident to $v$ are both in $M_1$; (ii) $3\times 6=18$ whose edges nonincident to $v$ are one in $M_1$ and one not in $M_1$, and; (iii) $\binom 62=15$ whose edges nonincident to $v$ are none in $M_1$.

Let us first consider $G=W^{M_1}_s$. Let $v$ be a vertex of its distillation $W$. As $v$ has degree $3$ in $W$ and $M_1$ is a perfect matching of $W$, $v$ is incident in $G$ to one chain of length $s$ and two chains of length $s+1$. Hence, if $F$ is a $5$-edge-cut of $W$ separating $v$, the number of $5$-edge-cuts of $G$ induced by $F$ is $s(s+1)^2\ell_1\ell_2$ where $\ell_1$ and $\ell_2$ are the lengths of the two chains of $G$ corresponding to the edges in $F$ nonincident to $v$. As each edge of $M_1$ corresponds to a chain of length $s$ in $G$ and each edge not in $M_1$ to a chain of length $s+1$ in $G$, the analysis in the preceding paragraph shows that the total number of $5$-edge-cuts of $G$ induced by $5$-edge-cuts of $W$ that separate $v$ is $s(s+1)^2(3s^2+18s(s+1)+15(s+1)^2)=36s^5+120s^4+147s^3+78s^2+15s$.

Recall from the proof of Lemma~\ref{lemma:count5} that the $5$-edges-cuts of $W$ separating two distinct vertices $v$ and $w$ are precisely those sets $F$ consisting of the edges incident to at least one of $v$ and $w$ for any edge $vw$ of $W$. As $v$ and $w$ have both degree $3$ in $W$ and $M_1$ is a perfect matching of $W$, the number of $5$-edge-cuts of $G$ induced by $F$ is $s(s+1)^4$ or $s^2(s+1)^3$, depending on whether $vw\in M_1$ or not, respectively. Since no $5$-edge-cut of $W$ can separate three distinct vertices, there number of $5$-edge-cuts of $G$ that separate two distinct vertices is $4s(s+1)^4+8s^2(s+1)^3=12s^5+40s^4+48s^3+24s^2+4s$.

By the inclusion-exclusion principle, we conclude that
\begin{align*}
   \mu_5^{\mathrm V}(W^{M_1}_{s})&=8\times(36s^5+120s^4+147s^3+78s^2+15s)-(12s^5+40s^4+48s^3+24s^2+4s)\\
                                 &=276s^5+920s^4+1128s^3+600s^2+116 s.
\end{align*}

Next, we study $G=W^X_{s}$ similarly. Let $v$ be a vertex of its distillation $W$. We consider different cases:
\begin{itemize}
\item If $v$ is neither $4$ nor $8$, then $v$ is incident in $G$ to one chain of length $s$ and two chains of length $s+1$. Hence, the reasoning used in the case where $G=W^{M_1}_s$ still applies to prove that the number of $5$-edge-cuts of $G$ induced by $5$-edge-cuts of $W$ that separate $v$ is $36s^5+120s^4+147s^3+78s^2+15s$. 

\item Since $4$ is incident in $G$ to three chains of length $s+1$ and among the chains nonincident to $4$ there are $4$ of length $s$ and $5$ of length $s+1$, the number of $5$-edge-cuts of $G$ induced by $5$-edge-cuts of $W$ that separate $4$ is $(s+1)^3\big(\binom 42s^2+4\times 5\times s(s+1)+\binom 52(s+1)^2\big)=36 s^5+148s^4+238s^3+186s^2+70s+10$.

\item Since $8$ is incident in $G$ to two chains of length $s$ and one chain of length $s+1$ and among the chains nonincident to $8$ there are $2$ of length $s$ and $7$ of length $s+1$, the number of $5$-edge-cuts of $G$ induced by $5$-edge-cuts of $W$ that separate $8$ is $s^2(s+1)\big(s^2+2\times 7\times s(s+1)+\binom 72(s+1)^2\big)=36s^5+92s^4+77s^3+21s^2$.
\end{itemize}

Recall that each $5$-edge-cut $F$ of $W$ separating two distinct vertices $v$ and $w$ consist precisely of all the edges incident to $v$ and $w$, where $vw$ is any edge of $W$. Thus, for each such set $F$, the $5$-edge-cuts of $G$ induced by $F$ are precisely those consisting of one edge from the chain $\gamma$ of $G$ corresponding to $vw$ and one edge from each of the four chains incident to $\gamma$. We consider different cases:
\begin{itemize}
\item If $vw$ is $23$ or $56$, then $\gamma$ has length $s$ and the four chains incident to it have length $s+1$. Hence, the number of $5$-edge-cuts induced by $F$ is $s(s+1)^4$.

\item If $vw$ is $18$ or $78$, then $\gamma$ has length $s$ and is incident to one chain of length $s$ and three chains of length $s+1$. Hence, the number of $5$-edge-cuts induced by $F$ is $s^2(s+1)^3$.

\item If $vw$ is $34$ or $45$, then $\gamma$ has length $s+1$ and is incident to one chain of length $s$ and three chains of length $s+1$. Hence, the number of $5$-edge-cuts induced by $F$ is $s(s+1)^4$.

\item If $vw$ is $12$, $15$, $26$, $37$, $48$, or $67$, then $\gamma$ has length $s+1$ and is incident to two chains of length $s$ and two chains of length $s+1$. Hence, the number of $5$-edge-cuts induced by $F$ is $s^2(s+1)^3$.
\end{itemize}
As no $5$-edge-cut of $W$ separates more than two distinct vertices, the total number of $5$-edge-cuts of $G$ induced by some $5$-edge-cut of $W$ separating two distinct vertices is $2s(s+1)^4+2s^2(s+1)^3+2s(s+1)^4+6s^2(s+1)^3=12s^5+40s^4+48s^3+24s^2+4s$.

Therefore, by the inclusion-exclusion principle, we conclude that
\begin{align*}
  \mu_5^{\mathrm V}(W^X_{s})&=6\times(36s^5+120s^4+147s^3+78s^2+15s)+(36 s^5+148s^4+238s^3+186s^2+70s+10)+\\
                            &\phantom{{}={}}+(36s^5+92s^4+77s^3+21s^2)-(12 s^5+40 s^4+48 s^3+24 s^2+4s)\\
                            &=276s^5+920s^4+1149s^3+651s^2+156s+10.\qedhere
\end{align*}
\end{proof}

\begin{lemma}\label{lemma:edge5}
For each positive integer $s$,
\begin{align*}
  \mu_5^{\mathrm E}(W^{M_1}_{s})&=84s^5+280s^4+368s^3+248s^2+92s+16\quad\qquad\text{and}\\
  \mu_5^{\mathrm E}(W^X_{s})&=84s^5+280s^4+356s^3+216s^2+64s+8.
\end{align*}
\end{lemma}
\begin{proof}
Let us consider first $G=W^{M_1}_s$. Recall from the proof of Lemma~\ref{lemma:count5} that each $5$-edge-cut $F$ of $W$ separating an edge $e$ consists precisely of the four edges incident to $e$ and an additional edge nonincident to $e$. Thus, the $5$-edge-cuts of $G$ induced by $F$ consist of one edge of each of the chains incident in $G$ to the chain $\gamma$ corresponding to $e$ and one edge from any of the chains of $G$ that are nonincident to $\gamma$. One the one hand, if $e\in M_1$, then the four chains incident to $\gamma$ have length $s+1$ and the chains nonincident to $\gamma$ have $12s+8-(s+4(s+1))=7s+4$ edges in total. Thus, the number of $5$-edge-cuts of $G$ induced by $5$-edge-cuts of $W$ separating $e$ is $(s+1)^4(7s+4)$. On the other hand, if $e\notin M_1$, then $\gamma$ is incident to two chains of length $s$ and two chains of lengths $s+1$ and the number of edges in chains nonincident to $\gamma$ is $12s+8-(2s+3(s+1))=7s+5$. Hence, the number of $5$-edge-cuts of $G$ induced by $5$-edge-cuts of $W$ separating $e$ is $s^2(s+1)^2(7s+5)$. As $W$ has $4$ edges in $M_1$ and $8$ not in $M_1$ and no $5$-edge-cut of $W$ separates two distinct edges, the total number of $5$-edges-cuts of $G$ induced by edge-separating $5$-edge-cuts of $W$ is $4(s+1)^4(7s+4)+8s^2(s+1)^2(7s+5)=84s^5+280s^4+368s^3+248s^2+92s+16$. As no edge-separating $5$-edge-cut of $W$ is vertex-separating, this is also the value of $\mu_5^{\mathrm E}(W^{M_1}_{s})$.

Next, we study $G=W^X_s$ in a similar fashion. Let $e$ be an edge of $W$ and let $\gamma$ be the chain of $G$ corresponding to $e$. We consider different cases:
\begin{itemize}
 \item If $e$ is $23$ or $56$, then $\gamma$ has length $s$, $\gamma$ is incident to four chains of length $s+1$, and the number of edges in chains nonincident to $\gamma$ is $12s+8-(s+4(s+1))=7s+4$.
 \item If $e$ is $18$ or $78$, then $\gamma$ has length $s$, $\gamma$ is incident to one chain of length $s$ and three chains of length $s+1$, and the number of edges in chains nonincident to $\gamma$ is $12s+8-(2s+3(s+1))=7s+5$.
 \item If $e$ is $34$ or $45$, then $\gamma$ has length $s+1$, $\gamma$ is incident one chain of length $s$ and three chains of length $s+1$, and the number of edges in chains nonincident to $\gamma$ is $12s+8-(s+4(s+1))=7s+4$.
 \item If $e$ is $12$, $15$, $26$, $37$, $48$, or $67$, then $\gamma$ has length $s+1$, $\gamma$ is incident to two chains of length $s$ and two chains of length $s+1$, and the number of edges in chains nonincident to $\gamma$ is $12s+8-(2s+3(s+1))=7s+5$.
\end{itemize}
We conclude the number of $5$-edge-cuts of $G$ induced by some edge-separating $5$-edge-cut of $W$ is $2(s+1)^4(7s+4)+2s(s+1)^3(7s+5)+2s(s+1)^3(7s+4)+6s^2(s+1)^2(7s+5)=84s^5+280s^4+356s^3+216s^2+64s+8$.
As no edge-separating $5$-edge-cut of $W$ is vertex-separating, this is also the value of $\mu_5^{\mathrm E}(W^X_{s})$.
\end{proof}

\begin{lemma}\label{lemma:path5}
For each positive integer $s$,
\begin{align*}
\mu_5^{P_3}(W^{M_1}_{s})&=24s^5+80s^4+104s^3+64s^2+16s\quad\qquad\text{and}\\
\mu_5^{P_3}(W^X_s)&=24s^5+80s^4+105s^3+69s^2+23s+3.
\end{align*}
\end{lemma}
\begin{proof}
Recall from the proof of Lemma~\ref{lemma:count5} that the $P_3$-separating $5$-edge-cuts $F$ of $W$ consist of the edges of $W$ sharing exactly one endpoint with at least one of $e_1$ and $e_2$, where $e_1$ and $e_2$ are any two incident edges of $W$. Moreover, $F$ determines $\{e_1,e_2\}$ univocally. Thus, if $G$ is a graph such that $D(G)=W$, then the number of $5$-edge-cuts induced by $F$ equals the product of the lengths of the five chains of $G$ having precisely one endpoint in common with $e_1$ or $e_2$.

Let us consider first $G=W^{M_1}_s$. Let $e_1$ and $e_2$ be two incident edges of $W$ and let $F$ be as in the preceding paragraph. On the one hand, if one of $e_1$ and $e_2$ belongs to $M_1$, then the other does not and, of the five chains of $G$ sharing exactly one endpoint with at least one of $e_1$ or $e_2$, there is one of length $s$ and four of length $s+1$. Thus, in this case, the number of $5$-edge-cuts of $G$ induced by $F$ is $s(s+1)^4$. On the other hand, if neither $e_1$ nor $e_2$ belongs to $M_1$, then, of the five chains sharing exactly one endpoint with at least one of $e_1$ or $e_2$, there are three of length $s$ and two of length $s+1$. Hence, in this case, the number of $5$-edge-cuts of $G$ induced by $F$ is $s^3(s+1)^2$. As $W$ has $16$ pairs of incident edges where one of them belongs to $M_1$ and $8$ pairs of incident edges such that none of them belongs to $M_1$, we conclude that $\mu_5^{P_3}(W^{M_1}_{s})=16s(s+1)^4+8s^3(s+1)^2=24s^5+80s^4+104s^3+64s^2+16s$.

Next, we analyze $W^X_s$ in similar way. Let $e_1$ and $e_2$ be two incident edges of $G$, let $F$ be as in the first paragraph of this proof, and let $H$ be the family of chains corresponding to the edges in $F$ (i.e., the chains of $G$ that share exactly one endpoint with at least one of $e_1$ and $e_2$). Precisely one of the following cases holds:
\begin{itemize}
 \item If $\{e_1,e_2\}$ is $\{23,34\}$, $\{45,56\}$, or $\{78,81\}$, then $H$ consists of five chains of length $s+1$.
 
 \item If $\{e_1,e_2\}$ is $\{12,23\}$, $\{56,67\}$, $\{23,37\}$, $\{32,26\}$, $\{56,62\}$, $\{65,51\}$, $\{78,84\}$, or $\{18,84\}$, then $H$ consists of $1$ chain of length $s$ and $4$ of length $s+1$.
 
 \item If $\{e_1,e_2\}$ is $\{34,45\}$, $\{67,78\}$, $\{81,12\}$, $\{43,37\}$, $\{45,51\}$, $\{87,73\}$, or $\{81,15\}$, then $H$ consists of $2$ chains of length $s$ and $3$ of length $s+1$.
 
 \item If $\{e_1,e_2\}$ is $\{12,26\}$, $\{21,15\}$, $\{34,48\}$, $\{54,48\}$, $\{67,73\}$, or $\{76,62\}$, then $H$ consists of $3$ chains of length $s$ and $2$ of length $s+1$.
\end{itemize}
As the number of $5$-edge-cuts of $G$ induced by $F$ equals the product of the lengths of the chains in $H$, it follows that $\mu_5^{P_3}(W^X_s)=3(s+1)^5+8s(s+1)^4+7s^2(s+1)^3+6s^3(s+1)^2=24s^5+80s^4+105s^3+69s^2+23s+3$.
\end{proof}

\begin{lemma}\label{lemma:match5}
For each positive integer $s$,
\begin{align}
\mu_5^{C_4}(W^{M_1}_{s})&=16s^5+44 s^4+64 s^3+56 s^2+24s+4\quad\qquad\text{and}\\
\mu_5^{C_4}(W^X_s) &=16s^5+44s^4+40s^3+12 s^2
\end{align}

\end{lemma}
\begin{proof}
Recall from the proof of Lemma~\ref{lemma:count5} that the induced subgraphs of $W$ isomorphic to $C_4$ are those induced by $D_1=\{1,2,6,5\}$, $D_2=\{2,3,7,6\}$, $D_3=\{3,4,8,7\}$, or $D_4=\{4,5,1,8\}$ and that the $5$-edge-cuts $F$ of $W$ that separate $D_i$ are those consisting of the four edges of $W$ incident to precisely one vertex of $D_i$ and one additional edge incident to no vertex of $D_i$. Notice that the family of $5$-edge-cuts of $W$ that separate $D_2$ or $D_4$ is precisely the family of sets that arise by adding to $M_1$ any edge of $W$ not in $M_1$. Similarly, if $M_1'=\{23,45,67,18\}$, then the family of $5$-edge-cuts of $W$ that separate $D_1$ or $D_3$ is precisely the family of sets that arise by adding to $M_1'$ any edge of $W$ not in $M_1'$. Thus, if $G$ is a graph such that $D(G)=W$, then the number of $5$-edge-cuts of $G$ induced by $C_4$-separating $5$-edge-cuts of $W$ is the sum of: (i) the product of the lengths of the chains of $G$ corresponding to the edges in $M_1$ times the number of edges of $G$ in chains not corresponding to edges in $M_1$ and; (ii) to the product of the lengths of the chain of $G$ corresponding to edges in $M_1'$ times the number of edges of $G$ in chains not corresponding to edges in $M_1'$. Therefore, by simple inspection, $\mu_5^{C_4}(W^{M_1}_{s})=s^4(12s+8-4s)+(s+1)^4(12s+8-4(s+1))=16s^5+44s^4+64s^3+56s^2+24s+4$ and $\mu_5^{C_4}(W^{X}_{s})=2s^2(s+1)^2(12s+8-2s-2(s+1))=16s^5+44s^4+40s^3+12s^2$.
\end{proof}

We now combine Lemmas~\ref{lemma:node5} to \ref{lemma:match5}, to show that $W^X_s$ is more reliable than $W^{M_1}_s$ near $\rho=1$.

\begin{proposition}\label{proposition:t-optimality}
For each positive integer $s$, $\mu_5(W^X_s)<\mu_5(W^{M_1}_{s})$. As a result, $W^X_{s}$ is more reliable than $W^{M_1}_{s}$ near $\rho=1$. 
\end{proposition}
\begin{proof} Recall that, if $G$ is $W^{M_1}_{s}$ or $W^X_s$, then \eqref{eq:muI} holds; i.e., $\mu_5^{\mathrm I}=\mu_5^{\mathrm V}+\mu_5^{\mathrm E}+\mu_5^{P_3}+\mu_5^{C_4}$. Thus, Lemmas~\ref{lemma:node5} to \ref{lemma:match5} imply that
\begin{align*}
  \mu_5^{\mathrm I}(W^{M_1}_s)&=400s^5+1324s^4+1664s^3+968s^2+248s+20\quad\qquad\text{and}\\
  \mu_5^{\mathrm I}(W^X_s)&=400s^5+1324s^4+1650s^3+948s^2+243s+21.
\end{align*}
Hence,
\begin{align*}
  \mu_5^{\mathrm I}(W^{M_1}_{s})-\mu_5^{\mathrm I}(W^X_s) &= 14s^3+20s^2+5s-1= 14(s-1)^3+62(s-1)^2+87(s-1)+38>0
\end{align*}
because $s$ is a positive integer. Thus, $\mu_5^{\mathrm I}(W^{X}_{s})<\mu_5^{\mathrm I}(W^{M_1}_s)$ and, by virtue of Proposition~\ref{prop:mu^I}, also $\mu_5(W^X_s)<\mu_5(W_s^{M_1})$.

We claim that, for each $k\in\{6,7,\ldots,12\}$, $\mu_k(W^{X}_s)=\mu_k(W^{M_1}_s)=\binom{12s+8}k$ because removing any set $F$ of $k$ edges from either $W^{M_1}_s$ or $W^X_s$ always gives a disconnected graph; in fact, the resulting graph has $12s+4$ vertices and fewer than $12s+3$ edges. Therefore, by applying Theorem~\ref{theorem:optlocal}\ref{it2:optlocal} with $j=5$, we conclude that $W^X_s$ is more reliable than $W^{M_1}_{s}$ near $\rho=1$.\end{proof}

We are ready to prove our main result.

\begin{theorem}\label{theorem:principal}
For each positive integer $s$, there is no UMRG in the class $\mathcal C_{12s+4,12s+8}$. 
\end{theorem}
\begin{proof} By Proposition~\ref{proposition:rho0}, $W^{M_1}_{s}$ is locally most reliable graph near $\rho=0$. Hence, $W_s^{M_1}$ is the only possible candidate in $\mathcal C_{12s+4,12s+8}$ to be a UMRG. However, by Proposition~\ref{proposition:t-optimality}, $W^X_s$ is more reliable than $W^{M_1}_{s}$ near $\rho=1$. This contradiction shows that there is no UMRG in $\mathcal{C}_{12s+4,12s+8}$. 
\end{proof}

\section*{\normalsize{Acknowledgments}}
P.\ Romero was partially supported by Comisi\'on Sectorial de Investigaci\'on Cient\'ifica (CSIC) and Project FCE-ANII entitled \emph{Teor\'ia Construcci\'on de Redes de M\'axima Confiabilidad}. M.D.\ Safe was partially supported by CONICET Grant PIBAA 28720210101185CO and Universidad Nacional del Sur Grant PGI 24/L115. The authors thank Dr. Guillermo Dur\'an for his permanent support.

\end{document}